\documentclass{amsart}

\usepackage[english]{babel}
\usepackage{amsfonts, amsmath, amsthm, amssymb,amscd,indentfirst}

\usepackage{MnSymbol}

\usepackage{tikz-cd}

\usepackage{scalerel}
\usepackage{stackengine,wasysym}

\usepackage{color}
\usepackage{xcolor}
\usepackage{tikz}
\usepackage{esint}

\usepackage{esint}
\usepackage{graphicx}
\usepackage{epstopdf}
\usepackage{amsmath,amssymb,latexsym,indentfirst}
\usepackage{times}
\usepackage{palatino}

\usepackage{stackengine}
\usepackage{scalerel}

\def\ba{\begin{align}}
	\def\ea{\end{align}}
\def\bp{\begin{proof}}
	\def\ep{\end{proof}}

\theoremstyle{plain}
\newtheorem{theorem}{Theorem}[section]

\newtheorem{proposition}[theorem]{Proposition}

\theoremstyle{definition}
\newtheorem{definition}[theorem]{Definition}

\newtheorem{remark}[theorem]{Remark}

	\def\bp{\begin{proof}}
		\def\ep{\end{proof}}

\begin{document}
		
		\title[A harmonic level sets proof of a positive mass theorem]{A harmonic level set proof of a positive mass theorem}
		
		\author{Rondinelle Marcolino Batista}
		\address{Universidade Federal do Piau\'{i} (UFPI), Departamento de Matem\'{a}tica, Campus Petr\^onio Portella, 64049-550, Teresina, PI, Brazil}
		\email{rmarcolino@ufpi.edu.br}
		\author{Levi Lopes de Lima}
		\address{Universidade Federal do Cear\'a (UFC),
			Departamento de Matem\'{a}tica, Campus do Pici, Av. Humberto Monte, s/n, Bloco
			914, 60455-760,
			Fortaleza, CE, Brazil.}
		\email{levi@mat.ufc.br}
		\thanks{
		 L.L. de Lima has been suported by 
			FUNCAP/CNPq/PRONEX 00068.01.00/15.}

		\begin{abstract}
		We provide a harmonic level set proof (along the lines of the argument in \cite{bray2022harmonic}) of the positive mass theorem for asymptotically flat $3$-manifolds with a non-compact boundary first established by Almaraz-Barbosa-de Lima in \cite{almaraz2014positive}.  
		\end{abstract}

		\maketitle

	\section{Introduction}
	
	To each $3$-dimensional asymptotically flat Riemannian manifold $(M,g)$, viewed as a (time-symmetric) initial data set of an isolated gravitational system in the context of General Relativity, the ADM construction associates an asymptotic invariant, the {\em ADM mass}, which may be viewed as the total mass of the system as computed by the appropriate version of the Hamiltonian formalism \cite{arnowitt1962gravitation, harlow2020covariant}. 
More precisely, if $g_{ij}$ are the coefficients of $g$ with respect to an asymptotically flat chart at infinity then 
	\begin{equation}\label{adm:mass}
		m_{ADM}(M,g)=\lim_{r\to +\infty}\frac{1}{16\pi}\int_{S^2_r}\left(g_{ij,j}-g_{jj,i}\right)\mu^i dS^2_r,
		\end{equation}
		where $\mu$ is the outer unit normal vector field to a large coordinate sphere $S^2_r$ of radius $r$ in the asymptotic region, the comma denotes partial differentiation and we use the Einstein convention of summing over repeated indexes here and elsewhere.
	Under suitable decay assumptions on the underlying matter fields, this ADM mass turns out to be a conserved quantity under the time propagation of the initial data set \cite{christodoulou2008mathematical,de2023conserved} and hence qualifies as a fundamental invariant of the associated gravitational system. 
	
	Physical reasoning demands that under a suitable dominant energy condition the ADM mass of $(M,g)$ should be non-negative and, moreover, its vanishing should imply that $(M,g)$ is isometric to the flat Euclidean space $(\mathbb R^3,\delta)$, the initial data set of Minkowski $4$-space. The difficulty in settling this assertion lies in the fact that, as it is apparent from (\ref{adm:mass}), the ADM mass is a flux integral at spacial infinity whose integrand does not seem to display any direct relationship to the scalar curvature, which in this time-symmetric case turns out to be the relevant energy density. Hence, any approach to this result should involve the consideration of a {\em global} object mediating between the ADM $1$-form $\omega_i=g_{ij,j}-g_{jj,i}$ and the scalar curvature.  
	Despite these inherent difficulties, a proof of the corresponding Positive Mass Theorem  (PMT) was eventually obtained by Schoen-Yau \cite{schoen1979proof} using minimal surfaces. Soon after that Witten was able to provide a simpler proof based on spinors \cite{witten1981new}. Yet another approach follows from the weak inverse mean curvature flow in \cite{huisken2001}. We should also mention the elementary approach in \cite{lam2011graph}, which applies to initial data sets which may be realized as graphs in Euclidean space; this approach has been recast in \cite{de2015adm}, where it has been shown that $\omega_i$ relates to the Newton tensor of the shape operator of the underlying graph. To this 
 (by no means exhaustive!) account of approaches to PMT we add the recent ``quasi-local'' polyhedron comparison proof following from the main result in \cite{li2020polyhedron}, which by its turn confirms a conjecture on dihedral rigidity due to Gromov \cite{gromov2014dirac}.

More recently, a new approach to PMT emerged from the seminal work by Stern \cite{stern2022scalar}. An argument leading to PMT eventually appeared in \cite{bray2022harmonic}; see \cite{bray2023spacetime} for an updated survey on this subject. 
As already pointed out in \cite{bray2022harmonic}, this reasoning somehow resembles the  proofs by Schoen-Yau
and Witten, with harmonic functions replacing spinors and their level sets playing the role of minimal surfaces. 
We remark that a similar proof, exploring the level sets of a certain Green's function, has appeared in \cite{agostiniani2021green}. 

On the other hand, motivated by questions related to the Yamabe problem for compact manifolds with boundary, it has been introduced in \cite{almaraz2014positive} a mass-type invariant for asymptotically flat $3$-manifolds with a {\em non-compact}
boundary $\Sigma$ modeled at infinity on Euclidean half-space $\mathbb R^3_+=\{x\in\mathbb R^3;x_3\geq 0\}$. 
	       
\begin{definition}\label{def:af:abdl}\cite{almaraz2014positive}
	A Riemannian $3$-manifold $(M,g)$ with a non-compact boundary $\Sigma$ is {\em asymptotically flat} with decay rate $\tau>1/2$ if there exists a compact subset $K\subset M$ and a diffeomorphism $\Psi:M\backslash K\to\mathbb R^3_+\backslash B^+_1(0)$ such that the following expansion holds as $r\to +\infty$:
	\begin{equation}\label{asym:exp:g}
		|g_{ij}(x)-\delta_{ij}|+r|g_{ij,k}(x)|+r^2|g_{ij,kl}(x)|=O(r^{-\tau}), 
		\end{equation}
		Here, $x=(x_1,x_2,x_3)$ is the (asymptotically flat) coordinate system induced by $\Psi$, $r=|x|$, $g_{ij}$ are the components of $g$ with respect to $x$, $B_1^+(0)=\{x\in\mathbb R^3_+;|x|\leq 1\}$ and $\delta$ is the standard flat metric. We also assume that $R_g$, the scalar curvature of $g$, and $H_g$, the mean curvature of the embedding $\Sigma \hookrightarrow M$, are both integrable. 
	\end{definition}	
	
	\begin{remark}\label{simplic}
		As it is immediate from Definition \ref{def:af:abdl}, we adopt throughout this paper the simplifying assumption that $M$ carries a {\em unique} asymptotically flat half-end $M\backslash K$. Also, we take $\Sigma$ to be connected.  
		\end{remark}
	
	It turns that we may attach to any such manifold an asymptotic invariant by means of an expression similar to (\ref{adm:mass}).
	
	\begin{definition}\label{def:mass:abdl}\cite{almaraz2014positive} Under the conditions above, the {\em mass} of $(M,g)$ is given by 
		\begin{equation}\label{def:mass:abdl:2}
			\mathfrak m_{(M,g)}=\lim_{r\to +\infty}\frac{1}{16\pi}\left\{\int_{S^2_{r,+}}\left(g_{ij,j}-g_{jj,i}\right)\mu^i dS^2_{r,+}+\int_{S^1_r}g_{\alpha 3}\vartheta^\alpha dS^1_r\right\},
			\end{equation}       
			where $S^2_{r,+}$ is a large coordinate hemisphere of radius $r$ with outer unit normal $\mu$ and $\vartheta$ is the outer unit co-normal to the circle $S^1_r=\partial S^2_{r,+}$, oriented as the boundary of the bounded region $\Sigma_r\subset\Sigma$. 
	\end{definition}

It has been shown in \cite{almaraz2014positive}	that the limit on the right-hand side of (\ref{def:mass:abdl:2}) exists and its value does not depend on the particular asymptotically flat coordinates chosen. Thus, $\mathfrak m_{(M,g)}$ is an invariant of the asymptotic geometry of $(M,g)$ which also may be viewed as a conserved quantity \cite{almaraz2019spacetime,harlow2020covariant,de2023conserved}. Moreover, under suitable energy conditions imposed both on $M$ and along $\Sigma$, a positive mass theorem holds true in this setting as well.

\begin{theorem}\label{pmt:res}\cite{almaraz2014positive}
		If $(M,g)$ is an asymptotically flat manifold with a non-compact boundary $\Sigma$  satisfying $R_g \geq 0$ and $H_g \geq   0$ then  there holds $\mathfrak m_{(M,g)}\geq 0$, with the equality occurring only if $(M,g)=(\mathbb R^3,\delta)$ isometrically. 
	\end{theorem}
	
	\begin{remark}\label{space-time}
		By means of an adaptation of Witten's spinorial method, a proof of a space-time version of 
		Theorem \ref{pmt:res} has been carried out in \cite{almaraz2019spacetime} under the selection of suitable dominant energy conditions both on the interior and along the boundary of the (not necessarily time-symmetric) initial data set which may be interpreted in terms of the Lagrangian formulation of General Relativity in the presence of a boundary; see \cite[Remark 2.7]{almaraz2019spacetime}. In the time-symmetric case considered here, these space-time energy conditions reduce to $R_g\geq0$ and $H_g\geq 0$, which are then justified on purely physical grounds.  
		\end{remark}

There exist by now at least five different proofs of Theorem \ref{pmt:res}. It is already shown in \cite{almaraz2014positive} that the original arguments by Schoen-Yau and Witten carry over to this setting. Also, in \cite[Proposition 4.1]{almaraz2014positive} it is proved that ``harmonically flat'' initial data sets are dense in the space of all asymptotically flat initial data sets satisfying the corresponding energy conditions. As explained there, after doubling a ``generic'' initial data set along the boundary, the proof gets reduced to applying a positive mass theorem for manifolds with corners \cite{miao2002positive}. Another proof using free boundary surfaces appears in \cite{chai2018positive} and the result also follows from the argument involving the weak free boundary inverse mean curvature flow leading to the Penrose-type inequality in \cite{koerber2019riemannian}. The purpose of this short note is to present yet another proof based on the harmonic level set method mentioned above. 
As in \cite{bray2022harmonic}, a key step in the proof is to pass from $M$ to $M_{\rm ext}$, the {\em exterior region} of $M$; for a precise description of this construction, see the discussion surrounding Theorem \ref{ext:reg} below. Also, we set $\Sigma_{\rm ext}=M_{\rm ext}\cap \Sigma$, the {\em boundary exterior region}. By using the analytical machinery developed in \cite{almaraz2014positive}, we are able to find a suitable harmonic function $u:M_{\rm ext}\to\mathbb R$ which will mediate between the energy densities and the integrands in (\ref{def:mass:abdl:2}); see Proposition \ref{anal:mach}. More precisely, after plugging this function into the integral inequality established in \cite[Proposition 4.2]{bray2022harmonic}, we obtain our main result, an explicit lower bound for the mass; compare with \cite[Theorem 1.2]{bray2022harmonic}. 
	
	\begin{theorem} \label{thm:main}
	If $(M^3,g)$ is an asymptotically flat manifold with a non-compact boundary $\Sigma$  satisfying $R_g \geq 0$ and $H_g \geq   0$ then 
	\begin{equation}\label{mass:est}
		\mathfrak{m}_{(M,g)}\geq\frac{1}{16\pi}\int_{M_{\rm ext}}\left(\frac{|\nabla^2 u|^2}{|\nabla u|}+ R_g|\nabla u|\right)dV 
		+\frac{1}{8\pi}\int_{\Sigma_{\rm ext}}H_g|\nabla u| dA,
	\end{equation}
	where $\nabla$ and $\nabla^2$ denote the gradient and Hessian operators of $g$.
\end{theorem}  

The proof of Theorem \ref{thm:main}, with an argument showing how it implies Theorem \ref{pmt:res}, is presented in Section \ref{proof}. 
	
	\section{The proof of Theorem \ref{thm:main}}\label{proof}
	
	Let $(M,g)$ be an oriented Riemannian $3$-manifold with a non-compact boundary $\Sigma$ and suppose that $(M,g)$ is asymptotically flat as in Definition \ref{def:mass:abdl}.
	As in the boundaryless case treated in \cite{bray2022harmonic}, the first step in the proof of Theorem \ref{thm:main} involves cutting out from $M$ a so-called {\em trapped region} $\mathcal T\subset M$. In our case, this construction has been carried out in \cite[Lemma 2.3]{koerber2019riemannian}, as we now recall. We start by considering the closure $\Gamma$ of the union of all smooth, immersed free boundary (with respect to $\Sigma$)
	and closed minimal surfaces, which turns out to be a compact subset of $M$. We then define $\mathcal T$ to be the union of $\Gamma$ and all compact components of $M\backslash \Gamma$, which is compact as well. Finally, we set $M_{\rm ext}$ to be the metric completion of $M\backslash \mathcal T$.  
	
	\begin{theorem}\label{ext:reg}\cite[Lemma 2.3]{koerber2019riemannian}
		Let $(M,g)$ be an asymptotically flat manifold with a connected non-compact boundary and assume that $R_g\geq 0$ and $H_g\geq 0$ everywhere. Then
		$M_{\rm ext}$ does not
		contain any other immersed minimal surfaces (free boundary or closed) and has the topology of a
		half-space with finitely many solid balls removed. Moreover, $\partial M_{\rm ext}\backslash \Sigma$ consists of finitely
		many free boundary minimal discs and closed minimal spheres.
		\end{theorem}
	
	\begin{remark}\label{simply}
		It follows that $M_{\rm ext}$ is simply connected and hence the Jordan-Brower separation property holds for any closed, embedded surface in $M_{\rm ext}$. 
		\end{remark}
		
		It follows from Theorem \ref{ext:reg} that 
		\[
		\partial M_{\rm ext}=\Sigma_{\rm ext}\sqcup 
		\mathcal S^{\rm cl}\sqcup \mathcal D^{\rm fb},
		\]
		where $\Sigma_{\rm ext}=\partial M_{\rm ext}\cap \Sigma$ is the {\em boundary exterior region}, 
		$\mathcal S^{\rm cl}=\sqcup_{j=1}^N S^{\rm cl}_j$ and $\mathcal D^{\rm fb}=\sqcup_{k=1}^{N'} D^{\rm fb}_k$, with
		each $ S^{\rm cl}_j$ and $D^{\rm fb}_k$ being a closed minimal surface and a free boundary minimal disk, respectively. Note that 
		$\Sigma_{\rm ext}\sqcup \mathcal D^{\rm fb}$ is connected and has the topology of $\mathbb R^2$. 
		We now fix an asymptotically flat coordinate system
	$x=(x_1,x_2,x_3)$ on $M_{\rm ext}$ so that near infinity $\Sigma_{\rm ext}$ is defined by $x_3=0$. In what follows, we denote by $\nu$ the outward unit normal vector field to $\Sigma_{\rm ext}$, which is then extended to $\partial M_{\rm ext}\backslash \Sigma_{\rm ext}$ in the obvious manner.  
	
	\begin{proposition}\label{anal:mach}
	Under the conditions above there exists	
	 a function $x'_3:M_{\rm ext}\to \mathbb R$ such that
	\begin{equation}\label{harmon}
		\left\{
		\begin{array}{rcl}
			\Delta x'_3&=&0\quad\mbox{in}\quad M_{\rm ext}\\
			x'_3&=& 0 \quad\mbox{on}\quad\Sigma_{\rm ext}\\
			\partial x_3'/\partial \nu& = & 0 \quad \mbox{on}\quad \mathcal S^{\rm cl}\sqcup \mathcal D^{\rm fb}
		\end{array}
		\right.
	\end{equation} 
		and
	\begin{equation}\label{infin}
		x'_3-x_3\in C^{2,\alpha}_{-\tau+1+\epsilon}(M),\quad \epsilon>0,
	\end{equation}
	as $r\to+\infty$.
	Moreover, $|\nabla x_3'|>0$ along $\Sigma_{\rm ext}$.
		\end{proposition}

		\begin{proof}
			We refer to \cite[Section 3]{almaraz2014positive} and \cite[Appendix C]{eichmair2023doubling} for the precise definition of the weighted H\"older space in (\ref{infin}). We adapt the argument outlined in the proof of \cite[Proposition 46]{eichmair2023doubling}, which relies on the analytical machinery developed in \cite{almaraz2014positive}. We first double $M_{\rm ext}$ across $\mathcal S^{\rm cl}\sqcup \mathcal D^{\rm fb}$ so as to form an asymptotically flat manifold $(\widehat M,\widehat g)$ with a non-compact boundary  $\widehat\Sigma$ and two half-ends on which a natural extension of $x_3$, say $\widehat x_3$, is defined. 
			We note that $\widehat g$ is smooth away from $\mathcal S^{\rm cl}\sqcup \mathcal D^{\rm fb}\hookrightarrow \widehat M$ and remains at least Lipschitz there. As a consequence, the coefficients of its Laplacian are at least Lipschitz as well (this uses that $\mathcal S^{\rm cl}\sqcup \mathcal D^{\rm fb}$ is formed by minimal surfaces).
			Using \cite[Proposition 3.8]{almaraz2014positive}, we may find a harmonic function $\widehat x_3':\widehat M\to\mathbb R$ satisfying  $\widehat x_3'|_{\widehat\Sigma}=0$ and such that $\widehat x_3-\widehat x'_3\in C^{2,\alpha}_{-\tau+1+\epsilon}(\widehat M)$. It is immediate that $x_3':=\widehat x_3'|_{M_{\rm ext}}$ meets all the conditions in (\ref{harmon}). Finally, the last assertion follows from the maximum principle. 
		\end{proof}

	It follows from (\ref{infin}) that $x'=(x_1',x_2',x_3')=(x_1,x_2,x_3')$ is an asymptotically flat coordinate system as well.
	For all $L>0$ large enough we set
		\begin{eqnarray*}
			D^+_L&=&\{x'_3=L; (x'_1)^2+(x'_2)^2\leq L^2\}\\
			T_L&=&\{0\leq x'_3\leq L; (x'_1)^2+(x'_2)^2= L^2\}\\
			D^0_L&=&\{x'_3=0; (x'_1)^2+(x'_2)^2\leq L^2\},
	\end{eqnarray*} 
	so that the coordinate half-cylinder $C_L=D^+_L\cup T_L$ lies in the asymptotic region.   
	Now, essentially the same argument leading to the geometric invariance of the mass in \cite{almaraz2014positive} implies that
	\begin{equation}\label{mass CL}
		\mathfrak{m}_{(M,g)} =\lim\limits_{L\rightarrow+\infty}\frac{1}{16\pi}\left\{\int_{C_L}\sum_i(g_{ji,i}-g_{ii,j})\nu^jdA +\int_{\partial D^0_L}g_{\alpha,3}\vartheta^{\alpha}ds\right\},
	\end{equation}
where here $\nu$ is the outward unit normal to $C_L=D^+_L\cup T_L$ and $\vartheta$ is the outward pointing unit co-normal to $\partial D^0_L$, oriented as the boundary of the bounded region $D^0_L\subset \Sigma$; this also follows from the general argument in \cite{michel2011geometric}. 

Another key ingredient in the sequel is the following integral inequality.
	
	\begin{proposition}\cite[Proposition 4.2]{bray2022harmonic}\label{integral}
		Let $(\Omega^3,g)$ be a 3-dimensional oriented compact Riemannian manifold with boundary $\partial\Omega=P_1\bigsqcup P_2$ and let  $u:\Omega\rightarrow\mathbb{R}$
		be a harmonic function satisfying $\partial u/\partial \nu=0$ on $P_1$ and $|\nabla_g u| > 0$ on $P_2$, where $\nu$ is outward unit normal to $\partial\Omega$. Then
		\begin{equation*}
			\int_{\underline{u}}^{\overline{u}}\int_{\Sigma_t}\frac{1}{2}\left(\frac{|\nabla^2 u|^2}{|\nabla u|^2}+ R_g-R_{\Sigma^u_t}+\int_{\partial\Sigma_t^u\cap P_1}H_{P_1}\right)dA_tdt 
			\leq  \int_{P_2}
			\frac{\partial|\nabla u|}{\partial\nu} dV, 
		\end{equation*}
		where $\Sigma_t^u=u^{-1}(t)$ with $t\in(\underline{u},\overline{u})$, the range of $u$, $dA_t$ is its area element, $R_{\Sigma^u_t}$ is the scalar curvature of $\Sigma^u_t$ and $H_{P_1}$ is the mean curvature of $P_1$. 
	\end{proposition}
		
		We may now start the proof of Theorem \ref{thm:main}. We apply Proposition \ref{integral} to $u=x'_3$ as in Proposition \ref{anal:mach}, $\Omega=\Omega_L\subset M$, the closure of the bounded component of $M_{\rm ext}\backslash S_L$, 
		where $S_L:=C_L\cup D_L^0\cup\mathcal S^{\rm cl}\cup \mathcal D^{\rm fb}$, 
		$\Sigma_t^L=\Omega_L\cap u^{-1}(t)$, 
		$P_1=\mathcal S^{\rm cl}\sqcup \mathcal D^{\rm fb}$, a disjoint union of minimal spheres and disks, and $P_2=C_L\cup D_L^0$. From Theorem \ref{ext:reg} we 
		find that 
		\begin{eqnarray*}
				\frac{1}{2}\int_{\Omega_L}\left(\frac{|\nabla^2 u|^2}{|\nabla u|}+ R_g|\nabla u|\right)dV 
				&\leq &\int_0^L\left(2\pi\chi(\Sigma_t^L)-\int_{\Sigma_t^L\cap T_L}
				\kappa_{t,L}\right)dt\\
				&& \quad +\int_{C_L}
				\frac{\partial|\nabla u|}{\partial\nu} dA+\int_{D^0_L}
				\frac{\partial|\nabla u|}{\partial\nu} dA, 
		\end{eqnarray*} 
		where $\chi$ denotes Euler characteristic,
	$\kappa_{t,L}$ is the geodesic curvature of the curve $\Sigma_t^L\cap T_L$ and we have used Gauss-Bonnet formula.
		
		We now claim that for any $t\in (0,L)$ a regular value of $u$,  $\Sigma_t^L$ has only one connected component and intersects $T_L$ along the circle $\Sigma_t^L\cap T_L$. Otherwise, there exists a regular value $t\in (0,L)$ and a component $\Sigma'\subset \Sigma_t^L$ disjoint from $T_L$. 
		By Remark \ref{simply}, there exists a domain $U\subset \Omega_L$ disjoint from $T_L$ and such that $\partial U\backslash\partial M_{\rm ext}=\Sigma'$. It then follows that $u\equiv t$ on $U$, which contradicts the fact that $t$ is a regular value. Thus, $\chi(\Sigma_t^L)\leq 1$ and we see that  
			\begin{eqnarray*}
				\frac{1}{2}\int_{\Omega_L}\left(\frac{|\nabla^2 u|^2}{|\nabla u|}+ R_g|\nabla u|\right)dV 
				&\leq &2\pi L-\int_0^L\Big(\int_{\Sigma_t^L\cap T_L}
				\kappa_{t,L}\Big)dt\\
				&&\quad +\int_{C_L}
				\frac{\partial|\nabla u|}{\partial\nu} dA+\int_{D^0_L}
				\frac{\partial|\nabla u|}{\partial\nu} dA.
		\end{eqnarray*}
	
		Next, since that $D^0_L\subset\Sigma=u^{-1}(0)$ is a regular level set we deduce that
		\begin{eqnarray*}
		H_g & = &  
		{\rm div}\left(-\frac{\nabla u}{|\nabla u|}\right)\\
		& = & -\frac{\Delta u}{|\nabla u|}+|\nabla u|^{-2}g\left(\nabla|\nabla u|,\nabla  u\right),
	\end{eqnarray*}
	where $\mbox{div}$ is the divergence, 
	so that 
	\[
	H_g=|\nabla u|^{-2}g\left(\nabla|\nabla u|,\nabla  u\right)
	\]
	along $D^0_L$. It follows that
	\[
	\frac{\partial|\nabla u|}{\partial\nu}=-|\nabla u|H_g,
	\]
		which immediately gives 
		\begin{eqnarray}\label{m}
				\frac{1}{2}\int_{\Omega_L}\left(\frac{|\nabla^2 u|^2}{|\nabla u|} + R_g|\nabla u|\right)dV 
				& \leq & -\int_{D^0_L}H_g|\nabla u| dA +2\pi L-\nonumber\\
				&   & - \int_0^L\left(\int_{\Sigma_t^L\cap T_L}
				\kappa_{t,L}\right)dt+\int_{C_L}
				\frac{\partial|\nabla u|}{\partial\nu} dA. 
		\end{eqnarray}
		
		We now note that $\nabla u=g^{3j}\partial_j$, where $\partial_j=\partial_{x_j'}$, and 
		\[
		\nabla|\nabla u|=\nabla(g^{33})^{1/2}=-\frac{1}{2}\nabla g_{33} +O(|x'|^{-1-2\tau}).
		\] 
		Moreover, the outer normal $\nu$ to $C_L=D^+_L\cup T_L$ is given by
		\[
		\nu=
		\left\{
		\begin{array}{rcc}
			\partial_3 + O(|x'|^{-\tau}) & \mbox{on} & D_L^+ \\
				\frac{x'_1\partial_1+x'_2\partial_2}{L}+ O(|x'|^{-\tau}) & \mbox{in} 
				        & T_L.
			\end{array}
		\right.
		\]
		Hence,
		\begin{eqnarray}\label{eq1}
			\int_{C_L}\frac{\partial|\nabla u|}{\partial\nu}dA 
			& = & -\frac{1}{2}\int_{D_L}g(\nabla g_{33},\partial_3)dA-\frac{1}{2L}\int_{T_L}g(\nabla g_{33},x'_1\partial_1+x'_2\partial_2)dA+\nonumber\\
			& & \quad +O(L^{1-2\tau})\nonumber\\
			& = & -\frac{1}{2}\int_{D^+_L}g_{33,3}dA-\frac{1}{2L}\int_{T_L}(x'_1g_{33,1}+x'_2g_{33,2})dA+O(L^{1-2\tau}).	
		\end{eqnarray}
		On the other hand, as $u=x'_3$ is harmonic it is easy to check that{\setlength\arraycolsep{2pt}\begin{eqnarray*}
				0&=&g^{ij}g(\nabla_{\partial_i}\nabla u, \partial_j)\\
				&=&g^{ij}g^{3l}g(\nabla_{\partial_i}\partial_l, \partial_j)+g^{ij}g_{lj}\partial_ig^{3l}\\
				&=&\Gamma^j_{j3}+\partial_lg^{3l}+O(|x|^{-1-2\tau})\\
				&=&\frac{1}{2}(g_{11,3}+g_{22,3}+g_{33,3})-g_{31,1}-g_{32,2}-g_{33,3}+O(|x|^{-1-2\tau})
		\end{eqnarray*}}
		which can be rewritten
		$$
		g_{33,3}=-2g_{31,1}-2g_{32,2}+g_{11,3}+g_{22,3}+O(|x|^{-1-2\tau}).
		$$
		Combined with \eqref{eq1} this yields
			\begin{eqnarray*}
				\int_{C_L}\frac{\partial|\nabla u|}{\partial\nu}dA&=&\int_{D^+_L}(g_{32,2}+g_{31,1}-\frac{1}{2} g_{22,3}-\frac{1}{2} g_{11,3})dA\\
				& & -\quad \frac{1}{2L}\int_{T_L}(x'_1g_{33,1}+x'_2g_{33,2})dA+O(L^{1-2\tau})\\
				&=&\frac{1}{2}\int_{D^+_L}(g_{31,1}-g_{11,3}+{g_{32,2}}- g_{22,3})dA+\frac{1}{2}\int_{D^+_L}(g_{31,1}+g_{32,2})dA\\
				&&\quad -\frac{1}{2L}\int_{T_L}(x'_1g_{33,1}+x'_2g_{33,2})dA+O(L^{1-2\tau})\\
				&=&\frac{1}{2}\int_{D^+_L}(g_{ji,i}-g_{ii,j})\nu^jdA+\frac{1}{2}\int_{D^+_L}
				\mbox{div}(\partial_3) dA\\
				&&\quad -\frac{1}{2L}\int_{T_L}(x'_1g_{33,1}+x'_2g_{33,2})dA+O(L^{1-2\tau}),
		\end{eqnarray*}
	and using the divergence theorem on $D_L^+$ and the fundamental theorem of calculus on $T_L$  we obtain
			\begin{eqnarray*}
				\int_{C_L}\frac{\partial|\nabla u|}{\partial\nu}dA
				&=&\frac{1}{2}\int_{D^+_L}(g_{ji,i}-g_{ii,j})\nu^jdA+\frac{1}{2L}\int_{\partial D^+_L}(x'_1g_{13}+x'_2g_{23})ds\\
				&&\quad -\frac{1}{2L}\int_{T_L}(x'_1g_{33,1}+x'_2g_{33,2})dA+O(L^{1-2\tau})\\
				&=&\frac{1}{2}\int_{D^+_L}(g_{ji,i}-g_{ii,j})\nu^jdA+\frac{1}{2L}\int_{T_L}
				\partial_3\left(x'_1g_{13}+x'_2g_{23}\right)dA\\
				&&\quad +\frac{1}{2L}\int_{\partial D^0_L}(x'_1g_{13}+x'_2g_{23})ds
				-\frac{1}{2L}\int_{T_L}(x'_1g_{33,1}+x'_2g_{33,2})dA+\\
				&&\quad \quad O(L^{1-2\tau}).
		\end{eqnarray*}
		We thus end up with 
			\begin{eqnarray}\label{eq2}
				\int_{C_L}\frac{\partial|\nabla u|}{\partial\nu}dA
				&=&\frac{1}{2}\int_{D^+_L}(g_{ji,i}-g_{ii,j})\nu^jdA\\\nonumber
				&&\quad +\frac{1}{2L}\int_{T_L}\left[x'_1(g_{13,3}-g_{33,1})+x'_2(g_{23,3}-g_{33,2})\right]dA\\\nonumber
				&&\quad\quad +\frac{1}{2}\int_{\partial D^0_L}g_{\alpha,3}\nu^{\alpha}ds+O(L^{1-2\tau}).
		\end{eqnarray}
		On the other hand, we may use \cite[Lemma 6.2]{bray2022harmonic} to deduce that
		{\setlength\arraycolsep{2pt}
			\begin{eqnarray}\label{kappa}
				\int_0^L\left(\int_{\Sigma_t^L\cap T_L}\kappa_{t,L}\right)dt&=&\frac{1}{2L}\int_{T_L}\left[x'_1(g_{22,1}-g_{12,2})+x'_2(g_{11,2}-g_{21,1})\right]dA\\\nonumber
				&&\quad +2\pi L+O(L^{1-2\tau}+L^{-\tau}).
		\end{eqnarray}}
		Consequently, it follows from \eqref{eq2} and \eqref{kappa} that
		{\setlength\arraycolsep{2pt}
			\begin{eqnarray*}
				-\int_0^L \left(\int_{\Sigma_t^L\cap T_L}\kappa_{t,L}\right)dt&+&\int_{C_L}\frac{\partial|\nabla u|}{\partial\nu}dA
				=\frac{1}{2}\int_{D^+_L}(g_{ji,i}-g_{ii,j})\nu^jdA\\\nonumber
				&&+\frac{1}{2L}\int_{T_L}\left[x'_1(g_{13,3}-g_{33,1})+x'_2(g_{23,3}-g_{33,2})\right]dA\\\nonumber
				&&+\frac{1}{2L}\int_{T_L}\left[x'_1(g_{12,2}-g_{22,1})+x'_2(g_{21,1}-g_{11,2})\right]dA\\\nonumber
				&&-2\pi L+\frac{1}{2}\int_{\partial D^0_L}g_{\alpha,3}\nu^{\alpha}ds+o(1)\\
				&=&-2\pi L+\frac{1}{2}\int_{C_L}(g_{ji,i}-g_{ii,j})\nu^j dA+\frac{1}{2}\int_{\partial D^0_L}g_{\alpha,3}\nu^{\alpha}ds\\&&+o(1).
		\end{eqnarray*}}
		Therefore, if we plug this into \eqref{m} we obtain
			\begin{eqnarray*}
				\frac{1}{2}\left[\int_{C_L}(g_{ji,i}-g_{ii,j})\nu^jdA+\int_{\partial D^0_L}g_{\alpha,3}\nu^{\alpha}ds\right]+o(1) &\geq & \frac{1}{2}\int_{\Omega_L}\left(\frac{|\nabla^2 u|^2}{|\nabla u|}+ R_g|\nabla u|\right)dV\\
				&&\quad +\int_{D^0_L}H_g|\nabla u| dA.
		\end{eqnarray*}
		and taking $L\rightarrow+\infty$ we achieve  \eqref{mass:est}.

		We now check how Theorem \ref{thm:main} implies Theorem \ref{pmt:res}. Clearly, \eqref{mass:est} implies $\mathfrak m_{(M,g)}\geq 0$  under the stated energy conditions. Also, if the equality $\mathfrak{m}_{(M,g)}=0$ holds then it follows from \eqref{mass:est} that $\nabla^2u\equiv 0$. Since the second fundamental form $B_{\Sigma_{\rm ext}}$ of the embedding $\Sigma_{\rm ext}\hookrightarrow M$ is proportional to $\nabla^2u$, $\Sigma_{\rm ext}$ is totally geodesic. 
		We now consider the $3$-manifold $(\widetilde M_{\rm ext},\widetilde g)$ obtained by doubling $(M_{\rm ext},g)$ across $\Sigma_{\rm ext}$, which is asymptotically flat (as a boundaryless manifold) but possibly carries finitely many inner horizons coming from the minimal spheres and free boundary minimal disks composing $\partial M_{\rm text}\backslash \Sigma_{\rm ext}$.
		Note that this is the kind of ``prepared'' initial data set treated in \cite[Section 6]{bray2022harmonic}. We now appeal to the Ashtekar-Hansen-type formulas in 
		\cite[Theorem 2.3]{herzlich2016computing} and \cite[Theorem 2.2]{de2019mass}, which due to their asymptotic  character can easily be shown to apply to this more general setting, namely, to 
		$(\widetilde M_{\rm ext},\widetilde g)$ and $(M_{\rm ext},g)$, respectively. Exploring the fact that $B_{\Sigma_{\rm ext}}=0$, it then follows from these formulas that $m_{ADM}(\widetilde M_{\rm ext},\widetilde g)=2\mathfrak m_{(M_{\rm ext},g)}=0$.
		We may now apply the argument in the last paragraph of \cite[Section 6.2]{bray2022harmonic} to conclude that no horizons exist and in fact $(\widetilde M,\widetilde g)=(\mathbb R^3,\delta)$ isometrically. It follows that $(M,g)=(\mathbb R^3_+,\delta)$ isometrically, which completes the proof of Theorem \ref{thm:main}.

\bibliographystyle{alpha}
\bibliography{pmt-ls}

	\end{document}